\documentclass[11pt,twoside,a4paper]{article}
\usepackage{color,amscd,amsmath,amssymb,amsthm,latexsym,stmaryrd,authblk,mathabx,shuffle,hyperref,tikz,tkz-tab} 
\usepackage[latin1]{inputenc}
\usepackage[T1]{fontenc}   
\usepackage[english]{babel}
\providecommand{\keywords}[1]{\textbf{\textit{Keywords.}} #1}
\providecommand{\AMSclass}[1]{\textbf{\textit{AMS classification.}} #1}

\input{xy}
\xyoption{all}

\title{Primitive elements of a connected free bialgebra}
\date{}
\author{Lo\"\i c Foissy}
\affil{\small{Univ. Littoral Côte d'Opale, UR 2597
LMPA, Laboratoire de Mathématiques Pures et Appliquées Joseph Liouville
F-62100 Calais, France}.\\ Email: \texttt{foissy@univ-littoral.fr}}

\theoremstyle{plain}
\newtheorem{theo}{Theorem}[section]
\newtheorem{lemma}[theo]{Lemma}
\newtheorem{cor}[theo]{Corollary}
\newtheorem{prop}[theo]{Proposition}
\newtheorem{defi}[theo]{Definition}

\theoremstyle{remark}
\newtheorem{remark}{Remark}[section]
\newtheorem{notation}{Notations}[section]
\newtheorem{example}{Example}[section]

\newcommand{\K}{\mathbb{K}}

\newcommand{\N}{\mathbb{N}}
\newcommand{\g}{\mathfrak{g}}
\newcommand{\h}{\mathfrak{h}}
\newcommand{\prim}{\mathrm{Prim}}

\newcommand{\gr}{\mathrm{Gr}}
\newcommand{\gs}{\mathbf{GS}}
\newcommand{\gfs}{\mathbf{GFS}}
\newcommand{\id}{\mathrm{Id}}
\newcommand{\calU}{\mathcal{U}}
\newcommand{\im}{\mathrm{Im}}
\renewcommand{\ker}{\mathrm{Ker}}

\begin{document}

\maketitle

\begin{abstract}
We prove that the Lie algebra of primitive elements of a graded and connected bialgebra, free as an associative algebra, over a field of characteristic zero,
is a free Lie algebra. The main tool is a filtration, which allows to embed the associated graded Lie algebra
into the Lie algebra of a free and cocommutative bialgebra. The result is then a consequence of Cartier--Quillen--Milnor--Moore's  Shirshov--Witt's theorems.
\end{abstract}


\AMSclass{16T05 16T10 17B01}

\tableofcontents

\section*{Introduction}

In the literature, several combinatorial bialgebras are known to be free as algebras: some in a direct way, by their very definition: for example, the bialgebra of noncommutative 
symmetric functions \cite{Hazewinkel2005,Gelfand1995}, bitensorial bialgebras \cite{Manchon1997}, bialgebras of planar rooted trees \cite{Foissy1,Foissy2}, 
of finite topologies, quasi-posets or double posets \cite{Foissy18,Foissy20,Foissy27,Malvenuto2011}, of graphs \cite{Foissy25,Holtkamp2003}$\ldots$; other ones in a more hidden way, 
often involving a notion of indecomposable objects:
the Loday--Ronco bialgebra of planar trees \cite{Aguiar2006,Loday1998,Ronco2000}, the Malvenuto--Reutenauer bialgebra of permutations \cite{Malvenuto1995}, 
the bialgebra of packed words \cite{Novelli2006-2} or of parking functions \cite{Foissy17,Novelli2006,Novelli2007}, and variants of them \cite{Aguiar2008,Bergeron2023,Catoire2022}, etc. 
To any of these bialgebras $H$ is attached the Lie algebra of its primitive elements $\prim(H)$, which structure is not so well-known in general.  
We proved in \cite{Foissy2} that the Lie algebra of primitive elements of the bialgebra of planar trees (and of decorated versions of it) is free. Using various isomorphisms,
this implies that this also the case for the Loday--Ronco or the Malvenuto--Reutenauer bialgebra \cite{Foissy2,Foissy5,Foissy15}. \\

Our aim here is to prove the following result:\\

\textbf{Main theorem (Theorem \ref{theo4.1})}. \emph{Let $H=(H,m,\Delta)$ be a graded, connected bialgebra over a field of characteristic zero, free as an algebra. 
Then the Lie algebra $\prim(H)$ is free.}\\ 

Here, graded means that the bialgebra $H$ is decomposed into a direct sum of vector spaces $H_n$, with $n\in \N$, such that:
\begin{enumerate}
\item For any $k,l\in \N$, $m(H_k\otimes H_l)\subseteq H_{k+l}$.
\item For any $n\in \N$, $\displaystyle \Delta(H_n)\subseteq \bigoplus_{p=0}^n H_p\otimes H_{n-p}$.
\end{enumerate}
Connected means that $H_0$ is 1-dimensional, generated by $1_H$. We do not assume that the subspaces $H_n$ are finite-dimensional.\\

A noticeable point is that all the bialgebras we consider here are in fact Hopf algebras, because of the connectedness condition. However, as the antipode won't play any role,
we chose to mention them everywhere as bialgebras and not Hopf algebras. \\

Here is an idea of the proof of this theorem.
When $H$ is cocommutative, Cartier-Quillen-Milnor-Moore's theorem implies that $H$ is the enveloping algebra $\calU(\g)$ of the Lie algebra of its primitive elements $\g$.
A classical result using the universal properties of $H$ as a free algebra and as an enveloping algebra proves that $\g$ is a free Lie algebra, see Proposition \ref{prop2.4}.
The situation is more complicated when $H$ is not cocommutative.  We shall use the notion of graded and filtered bialgebra (Definition \ref{defi1.7}).
To such an object $H$ is associated another graded bialgebra $\gr(H)$, by a functor from the category of graded filtered spaces to the category of filtered spaces. 
When $H$ is graded and connected, we can provide it the filtration given by the powers of the augmentation ideal of $H$, that is to say the kernel of the counit of $H$ (Proposition \ref{prop3.1}).
We prove that when $H$ is a free algebra, $\gr(H)$ is also free and cocommutative (Propositions \ref{prop3.2} and \ref{prop3.3}), 
which implies that the Lie algebra $\prim(\gr(H))$ is free as a Lie algebra. It turns out that the Lie algebra $\prim(H)$ is also filtered and 
that the associated graded Lie algebra $\gr(\prim(H))$ is isomorphic to a Lie subalgebra of $\prim(\gr(H))$: By Shirshov--Witt's theorem \cite{Shirshov53,Witt56}, 
$\gr(\prim(H))$ is a free Lie algebra. Then, coming back to $\prim(H)$, a reasoning using both the graduation and the filtration proves that $\prim(H)$ is also free. \\

The paper is structured in the following way:
the first section introduces graded filtered spaces (Definition \ref{defi1.5}), the functor $\gr$ from graded filtered spaces to graded spaces (Definition \ref{defi1.8} and Proposition \ref{prop1.9}) 
and proves that it is compatible with the tensor product (Proposition \ref{prop1.11}). The second section is devoted to various results on free graded algebras and enveloping algebras. 
Some of the results exposed there are rather classical (Propositions \ref{prop2.3} and \ref{prop2.4}, for example), but we choose to propose a proof for the sake of completeness.
In the third section, we introduce the counital filtration of any graded connected bialgebra (Proposition \ref{prop3.1}). 
We prove that the functor $\gr$ applied to this filtration preserves the freeness (Proposition \ref{prop3.2}), but changes the coproduct into a cocommutative one (Proposition \ref{prop3.3}). 
We also provs in this section several results on the Lie algebra of primitive elements of the associated graded bialgebra (Lemma \ref{lem3.4} and Proposition \ref{prop3.5}). 
Then, all the preliminary results are ready to prove the main theorem (Theorem \ref{theo4.1}), which proof is the object of the last section. \\

\textbf{Acknowledgements}. 
The author acknowledges support from the grant ANR-20-CE40-0007 \emph{Combinatoire Algébrique, Résurgence, Probabilités Libres et Opérades}.\\ 

\begin{notation}
$\K$ is a field of characteristic zero. All the vector spaces in this text will be taken over $\K$.
We refer to \cite{Abe1980,Cartier2021,Kashaev2023,Sweedler1969} for classical results and notations on Hopf algebras and bialgebras. 
\end{notation}

\section{Filtered graded objects}

\subsection{Graded spaces}

\begin{defi}
\begin{enumerate}
\item A graded space is a pair $V=(V,(V_n)_{n\in \N})$ such that $V$ is a vector space and $V_n$ is a subspace of $V$ for any $n\in \N$, with
\[V=\bigoplus_{n\in \N} V_n.\]
\item Let $V$ and $W$ be two graded spaces. A graded morphism $\phi$ from $V$ to $W$ is a linear map $\phi:V\longrightarrow W$ such that for any $n\in \N$, $\phi(V_n)\subseteq W_n$.
\end{enumerate}
\end{defi}

With these notions, graded spaces form a category denoted by $\gs$. 

\begin{prop}
Let $V$ and $W$ be two graded spaces. Then the vector space $V\otimes W$ is graded with, for any $n\in \N$,
\begin{align*}
(V\otimes W)_n&=\bigoplus_{i=0}^n V_i \otimes W_{n-i}.
\end{align*}
\end{prop}

\begin{proof}
Immediate verifications.
\end{proof}

One can easily verify that with this tensor product, $\gs$ is a symmetric monoidal category. It is then possible to consider algebras, Lie algebras, coalgebras, bialgebras$\ldots$
in this category. Let us explicit these notions:

\begin{defi} 
\begin{enumerate}
\item A graded algebra is an algebra $(A,m)$ together with a graduation $(A_n)_{n\in \N}$, such that $1_A\in A_0$ and for any $m,n\in \N$,
\[m(A_m\otimes A_n)\subseteq A_{m+n}.\]
\item A graded coalgebra is a coalgebra $(C,\Delta)$ together with a graduation $(C_n)_{n\in \N}$ such that for any $n\geq 1$, $\varepsilon_C(C_n)=(0)$
and for any $n\in \N$, 
\[\Delta(C_n)\subseteq \sum_{i=0}^n C_i\otimes C_{n-i}.\]
\item A graded bialgebra is a bialgebra $(H,m,\Delta)$ together with a graduation $(H_n)_{n\in \N}$, making it both a graded algebra and a graded coalgebra.
We shall say that $H$ is connected if $H_0$ is 1-dimensional (and so generated as a vector spaces by the unit $1_H$ of $H$).
\item A graded Lie algebra is a Lie algebra $(\g,[-,-])$ together with a graduation $(\g_n)_{n\in \N}$ such that for any $m,n\in \N$,
\[[\g_m,\g_n]\subseteq \g_{m+n}.\]
We shall say that it is connected if $\g_0=(0)$.
\end{enumerate}
\end{defi}

\begin{example}\begin{enumerate}
\item  The Lie algebra $\prim(H)$ of primitive elements of a connected bialgebra $H$ is a graded Lie algebra, with, for any $n\in \N$,
\[\prim(H)_n=\prim(H)\cap H_n.\]
If $H$ is connected, then $\prim(H)$ is connected.
\item If $\g$ is a graded Lie algebra, then its enveloping algebra $\cal(\g)$ is a graded bialgebra, with the graduation given by
\[\calU(\g)_n=\mathrm{Vect}(x_1\ldots x_p\mid p\geq 0,\: x_i\in \g_{n_i},\:n_1+\ldots+n_p=n).\]
Moreover, $\calU(\g)$ is connected if and only if, $\g$ is connected. 
\end{enumerate}\end{example}

We recall this important result \cite{Cartier2021,MilnorMoore65}, called Cartier--Quillen--Milnor--Moore's theorem:

\begin{theo}
Let $H$ be a graded, connected, cocommutative bialgebra. Then, as a graded bialgebra, $H$ is isomorphic to the enveloping algebra of the Lie algebra $\prim(H)$ of its primitive elements. 
\end{theo}

\subsection{Graded filtered spaces}

\begin{defi} \label{defi1.5}
A graded filtered space is a triple $V=(V,(V_n)_{n\in \N}, (V^{(k)})_{k\in \N})$ such that:
\begin{itemize}
\item $(V,(V_n)_{n\in \N})$ is a graded space.
\item $(V,(V^{(k)})_{k\in \N})$ is a filtered space.
\item For any $k\in \N$, $\displaystyle V^{(k)}=\bigoplus_{n\in \N} V_n\cap V^{(k)}$.
\end{itemize}
In the sequel, we put $V^{(k)}_n=V_n\cap V^{(k)}$. We shall say that the graded filtered space $V$ is locally finite if for any $n\in \N$, there exists $K(n)\in \N$
such that $V^{(K(n))}_n=(0)$.
\end{defi}

\begin{remark}
If $V$ is a locally finite filtered graded space, then for any $k,n\in \N$, as the family $(V_n^{(k)})_{k\in \N}$ decreases, 
\[k\geq K(n)\Longrightarrow V^{(k)}_n=(0).\]
\end{remark}

\begin{defi}
 Let $V$ and $W$ be two graded filtered spaces. A graded filtered morphism $\phi$ from $V$ to $W$ is a linear map $\phi:V\longrightarrow W$ such that for any $n\in \N$, 
$\phi(V_n)\subseteq W_n$ and for any $k\in \N$, $\phi(V^{(k)})\subseteq W^{(k)}$.
\end{defi}

With these notions, filtered spaces form a category denoted by $\gfs$. Let us consider algebras, coalgebras, bialgebras and Lie algebras in this category.

\begin{defi} \label{defi1.7}
\begin{enumerate}
\item A graded filtered algebra is an algebra $(A,m)$ together with a graduation $(A_n)_{n\in \N}$ and a filtration $(A^{(k)})_{k\in \N}$ making it a graded filtered space, 
such that $1_A\in A_0^{(0)}$ and for any $k,l,m,n\in \N$,
\[m\left(A^{(k)}_m\otimes A^{(l)}_n\right)\subseteq A^{(k+l)}_{m+n}.\]
\item A graded filtered coalgebra is a coalgebra $(C,\Delta)$ together with a graduation $(C_n)_{n\in \N}$ and a filtration $(C^{(k)})_{k\in \N}$ making it a graded filtered space, 
such that for any $(k,n)\neq (0,0)$, $\varepsilon_C\left(C^{(k)}_n\right)=(0)$
and for any $k,n\in\N$, 
\[\Delta\left(C^{(k)}_n\right)\subseteq \sum_{i=0}^k\sum_{j=0}^n C^{(i)}_j\otimes C^{(k-i)}_{n-j}.\]
\item A graded filtered bialgebra is a bialgebra $(H,m,\Delta)$ together with a graduation $(H_n)_{n\in \N}$ and a filtration $(H^{(k)})_{k\in \N}$ making it a graded filtered space, 
making it both a graded filtered algebra and a graded filtered coalgebra. We shall say that $H$ is connected if $H_0$ is 1-dimensional.
\item A graded Lie algebra is a Lie algebra $(\g,[-,-])$ together with a graduation $(\g_n)_{n\in \N}$ and a filtration $(\g^{(k)})_{k\in \N}$ making it a graded filtered space,
such that for any $k,l,m,n\in \N$,
\[\left[\g^{(k)}_m,\g^{(l)}_n\right]\subseteq \g^{(k+l)}_{m+n}.\]
We shall say that it is connected if $\g_0=(0)$.
\end{enumerate}
\end{defi}

\subsection{From graded filtered to graded objects}

\begin{defi}\label{defi1.8}
Let $V$ be a graded filtered space. For any $n\in \N$, we put $\displaystyle \gr(V)_n=\bigoplus_{k=0}^\infty \frac{V_n^{(k)}}{V_n^{(k+1)}}$
and $\displaystyle \gr(V)=\bigoplus_{n=0}^\infty \gr(V)_n$. 
\end{defi}

\begin{prop}\label{prop1.9}
Let $V$ and $W$ be two graded filtered spaces and let $\phi:V\longrightarrow W$ be a graded filtered morphism. We define $\gr(\phi):\gr(V)\longrightarrow \gr(W)$
by the following:
\begin{align*}
&\forall x\in V^{(k)},& \gr(\phi)\left(\pi_V^{(k)}(x)\right)&=\pi_W^{(k)}\circ \phi(x),
\end{align*}
where $\pi_X^{(k)}$ denotes the canonical surjection from $X^{(k)}$ to $\dfrac{X^{(k)}}{X^{(k+1)}}$, for any graded filtered space $X$.
Then $\gr(\phi)$ is a graded morphism.
\end{prop}

\begin{proof}
Let us assume that $\pi_X^{(k)}(x)=\pi_X^{(k)}(y)$. Then $x-y\in V^{(k+1)}$. As $\phi$ is graded filtered, $\phi(x-y)\in W^{(k+1)}$, so $\pi_W^{(k)}\circ \phi(x-y)=0$,
and finally $\pi_W^{(k)}\circ \phi(x)=\pi_W^{(k)}\circ \phi(y)$. We proved that $\gr(\phi)$ is well-defined. Let $x\in V_n^{(k)}$. As $\phi$ is graded, $\phi(x)\in W_n$,
so $\gr(\phi)\left(\pi_V^{(k)}(x)\right)\in \gr(W)_n$: $\gr(\phi)$ is graded.
\end{proof}

Obviously, if $V,W,X$ are three graded filtered spaces and $\phi:V\longrightarrow W$, $\psi:W\longrightarrow X$, then $\gr(\psi\circ \phi)=\gr(\psi)\circ \gr(\phi)$.
Moreover, $\gr(\id_{V})=\id_{\gr(V)}$, so:

\begin{prop}
$\gr:\gfs\longrightarrow \gs$ is a functor.
\end{prop}

\begin{prop}\label{prop1.11}
Let $V$ and $W$ be two graded spaces. Then the vector space $V\otimes W$ is graded, with, for any $n\in \N$,
\[(V\otimes W)_n=\bigoplus_{i=0}^n V_i \otimes W_{n-i}.\]
\end{prop}

\begin{proof}
Immediate verifications.
\end{proof}

\begin{prop}
Let $V$ and $W$ be two graded filtered spaces. Then the vector space $V\otimes W$ is graded filtered, with, for any $k,n\in \N$,
\begin{align*}
(V\otimes W)_n&=\bigoplus_{i=0}^n V_i \otimes W_{n-i},&(V\otimes W)^{(k)}&=\sum_{i=0}^k V^{(i)}\otimes W^{(k-i)}.
\end{align*}
\end{prop}

\begin{proof}
Immediate verifications.
\end{proof}

\begin{prop}
\begin{enumerate}
\item Let $V,W$ be two graded filtered spaces. Then, as graded spaces,
\[\gr(V\otimes W)=\gr(V)\otimes \gr(W).\]
\item Let $V,W,V',W'$ be four graded filtered spaces, $f:V\longrightarrow V'$ and $g:W\longrightarrow W'$ be two graded connected maps. Then
\[\gr(f\otimes g)=\gr(f)\otimes \gr(g).\]
\end{enumerate}
\end{prop}

\begin{proof}
1. By construction, 
\begin{align*}
(\gr(V)\otimes \gr(W))^{(k)}&=\bigoplus_{i=0}^k \frac{V^{(i)}}{V^{(i+1)}}\otimes \frac{W^{(k-i)}}{W^{(k-i+1)}}.
\end{align*}
Let $x\in V^{(i)}$ and $y\in W^{(k-i)}$. Then $x\otimes y\in (V\otimes W)^{(k)}$. Moreover, if $x-x'\in V^{(i+1)}$ and $y-y'\in W^{(k-i+1)}$, then
\[x\otimes y-x'\otimes y'=(x-x')\otimes y+x'\otimes (y-y')\in V^{(i+1)}\otimes W^{(k-i)}+V^{(i)}\otimes W^{(k-i+1)}\subseteq (V\otimes W)^{(k+1)}.\]
Therefore, we define a linear map
\begin{align*}
\theta:&\left\{\begin{array}{rcl}
\gr(V)\otimes \gr(W)&\longrightarrow&\gr(V\otimes W)\\
\pi_V^{(i)}(x)\otimes \pi_W^{(k-i)}(y)&\longmapsto&\pi_{V\otimes W}^{(k)}(x\otimes y).
\end{array}\right.
\end{align*}
Let us prove that $\theta$ is an isomorphism. For any $k\in \N$, let us choose a complement 
$V^{(=k)}$ of $V^{(k+1)}$ in $V^{(k)}$ and $W^{(=k)}$ of $W^{(k+1)}$ in $W^{(k)}$. 
The following maps are then isomorphisms:
\begin{align*}
\psi_V&:\left\{\begin{array}{rcl}
\displaystyle \bigoplus_{k\in \N}V^{(=k)}&\longrightarrow& \gr(V)\\
v\in V^{(=k)}&\longmapsto&\pi_V^{(k)}(v),
\end{array}\right.&
\psi_W&:\left\{\begin{array}{rcl}
\displaystyle \bigoplus_{l\in \N}W^{(=l)}&\longrightarrow& \gr(W)\\
w\in W^{(=l)}&\longmapsto&\pi_W^{(l)}(w),
\end{array}\right.&
\end{align*}
and the following map is an isomorphism:
\[\theta_1=\psi_V\otimes \psi_W:
\left\{\begin{array}{rcl}
\displaystyle \bigoplus_{k,l\in \N} V^{(=k)}\otimes W^{(=l)}&\longrightarrow&\gr(V)\otimes \gr(W)\\
v\otimes w\in V^{(=k)}\otimes W^{(=l)}&\longmapsto&
\pi_V^{(k)}(v)\otimes \pi_W^{(l)}(w).
\end{array}\right.\]
Let $n\in \N$. Let us prove that
\[\left(\bigoplus_{k+l=n} V^{(=k)}\otimes W^{(=l)}\right)\oplus
(V\otimes W)^{(n+1)}=W^{(n)}.\]
Any element of $W^{(n)}$ is a linear span of elements $x\otimes y$, with $x\in V^{(k)}$, $y\in V^{(l)}$, with $k+l=n$. 
Moreover, for such a tensor, we can write
\begin{align*}
x&=x_1+x_2,&x_1&\in V^{(=k)}, &x_2&\in V^{(k+1)},\\
y&=y_1+y_2,&y_1&\in W^{(=l)}, &y_2&\in W^{(l+1)},
\end{align*}
so
\[x\otimes y=\underbrace{x_1\otimes y_1}_{\in V^{(=k)}\otimes W^{(=l)}}
+\underbrace{x_1\otimes y_2+x_2\otimes y_1+x_2\otimes y_2}_{\in (V\otimes W)^{(n+1)}},\]
which gives
\[\left(\bigoplus_{k+l=n} V^{(=k)}\otimes W^{(=l)}\right)+ (V\otimes W)^{(n+1)}=W^{(n)}.\]
Let us now assume that
\[\left(\bigoplus_{k+l=n} V^{(=k)}\otimes W^{(=l)}\right)\cap (V\otimes W)^{(n+1)}\neq (0).\]
Let us consider a nonzero element $X$ in this intersection. Let us assume that the component $X_{k,l}$ 
of $X$ in $V^{(=k)}\otimes W^{(=l)}$ is nonzero. We write it as a sum of tensors
\[X_{k,l}=\sum_{i=1}^p v_i\otimes w_i,\]
with $p$ minimal. Because of this minimality, the families $(v_i)$ and $(w_i)$ are linearly independent:
hence, we can find $f$ in the dual of $V^{(=k)}$ and $g\in W^{(=l)}$ such that for any $i$,
\[f(v_i)=g(w_i)=\delta_{1,i}.\]
These maps $f$ and $g$ are extended to $V$ such that $f(V^{(k+1)})=g(W^{(l+1)})=(0)$. Then
\[f\otimes g(X_{k,l})=\sum_{i}f(v_i)g(w_j)+0=1.\]
As $X_{k,l}\in (V\otimes W)^{(n+1)}$, we can write it under the form
\[X_{k,l}=\sum_j x'_j\otimes y'_j,\]
with for any $j$, $x'_j \in V^{(k_j)}$, $y'_j\in V^{(l_j)}$, with $k_j+l_j>n$. Therefore, for any $j$, $k_j>k$ or $l_j>l$,
so $f\otimes g(x'_j\otimes y'_j)=0$ and finally $f\otimes g(X_{k,l})=0$: this is a contradiction. \\

As a consequence, we obtain an isomorphism
\[\theta_2:\left\{\begin{array}{rcl}
\displaystyle \bigoplus_{n\in \N}\left(\bigoplus_{k+l=n} V^{(=k)}\otimes W^{(=l)}\right)
=\bigoplus_{k,l\in \N} V^{(=k)}\otimes W^{(=l)}&\longrightarrow&\gr(V\otimes W)\\
v\otimes w\in V^{(=k)}\otimes W^{(=l)}&\longmapsto&\pi_{V\otimes W}^{(k+l)}(v\otimes w).
\end{array}\right.\]

We obtain the following commutative diagram:
\[\xymatrix{\gr(V)\otimes \gr(W)\ar[rr]^{\theta}&&\gr(V\otimes W)\\
&\displaystyle \bigoplus_{k,l\in \N} V^{(=k)}\otimes W^{(=l)}\ar[lu]^{\theta_1}\ar[ru]_{\theta_2}
}\]
As $\theta_1$ and $\theta_2$ are isomorphisms, $\theta=\theta_2\circ \theta_1^{-1}$ is an isomorphism. 
We now identify $\gr(V\otimes W)$ and $\gr(V)\otimes \gr(W)$ through this natural isomorphism. \\

2. Immediate verifications.
\end{proof}

Consequently, the functor $\gr$ is compatible with the tensor products. Consequently, if we have a graded filtered algebra (or Lie algebra,  bialgebra$\ldots$) $V$, 
then $\gr(V)$ is a graded algebra (or Lie algebra, bialgebra$\ldots$).

\section{Free graded connected algebras and enveloping algebras}

\subsection{Free graded algebras}

\begin{defi}
Let $A$ be a graded connected algebra. Its augmentation  ideal is $\displaystyle A_+=\bigoplus_{n=1}^\infty A_n$. It is an ideal of codimension 1.
\end{defi}

\begin{prop} \label{prop2.2}
Let $A$ be a graded connected algebra, free as a graded algebra. Let $V$ be a graded subspace of $A$ such that
$A_+=V\oplus A_+^2$. Then $A$ is isomorphic to $T(V)$ as a graded algebra. 
\end{prop}

\begin{proof}
Let $W$ be a graded subspace of $W$ such that $W$ freely generates $A$.
As $W$ is graded and $A_0=\K  1_A$, necessarily $W_0=(0)$, so $W\subseteq A_+$. Let us chose a basis $(w_i)_{i\in I}$ of $W$. Then:
\begin{itemize}
\item A basis of $A$ is given by $(w_{i_1}\ldots w_{i_k})_{k\in\geq 0,\: i_1,\ldots,i_k\in I}$.
\item A basis of $A_+$ is given by  $(w_{i_1}\ldots w_{i_k})_{k\geq 1,\: i_1,\ldots,i_k\in I}$.
\item A basis of $A_+^2$ is given by  $(w_{i_1}\ldots w_{i_k})_{k\geq 2,\: i_1,\ldots,i_k\in I}$.
\end{itemize}
Consequently, $A_+=W\oplus A_+^2$. As $A_+=V\oplus A_+^2$, we obtain a linear isomorphism $\phi:W\longrightarrow V$ such that
for any $x\in W$, $x-\phi(x)\in A_+^2$. As $V$, $A_+^2$ and $W$ are graded, $\phi$ is graded. 
For any $i\in I$, we put $v_i=\phi(w_i)$. Let us prove that $(v_{i_1}\ldots v_{i_p})_{p\geq 0, \: i_1,\ldots,i_p\in I}$
is a basis of $A_+$.

Let us firstly assume that there exists a family $(a_{i_1,\ldots,i_p})_{p\geq 0,\: i_1,\ldots,i_p\in I}$, with a finite (but non zero)
number of elements which are non zero, such that
\[\sum_{p\in \N}\sum_{i_1,\ldots,i_p\in I} a_{i_1,\ldots,i_p} v_{i_1}\ldots v_{i_p}=0.\]
Let us denote by $m$ the minimal $p$ such that there exists a non zero coefficient $a_{i_1,\ldots,i_p}$. As $w_i-v_i\in A_+^2$ for any $i$,
we obtain that 
\begin{align*}
\sum_{p\in \N}\sum_{i_1,\ldots,i_p\in I} a_{i_1,\ldots,i_p} v_{i_1}\ldots v_{i_p}
&=\sum_{i_1,\ldots,i_m}a_{i_1,\ldots,i_m} w_{i_1}\ldots w_{i_m}\\
&+\mbox{a span of monomials $w_{i_1}\ldots w_{i_p}$ with $p>m$}.
\end{align*}
As $(w_{i_1}\ldots w_{i_p})_{p\geq 0,\: i_1,\ldots,i_p\in I}$ is a basis of $A$, we obtain that for any $i_1,\ldots,i_m\in I$, $a_{i_1,\ldots,i_m}=0$,
which contradicts the definition of $m$. Therefore, $(v_{i_1}\ldots v_{i_p})_{p\geq 0,\: i_1,\ldots,i_p\in I}$ is linearly independent. 

Let us now prove that it is a generating family of $A$. It is equivalent to show that $V$ generates $A$ as an algebra.
We denote by $A'$ the subalgebra of $A$ generated by $V$ and let us prove that for any $a\in A_n$, $a\in A'$ by induction on $n$.
It is trivial if $n=0$, as $A_0=\K 1_A$. Let us assume the result at all ranks $<n$. As $A_+=V\oplus A_+^2$,
we can write $\displaystyle a=v+\sum_{i=1}^p a_i b_i$, with $v\in V$, and for any $i$, $a_i,b_i\in A_+$. 
By homogeneity of the product, we can assume that $v\in V_n$ and that for any $i$,
$a_i$ and $b_i$ are homogeneous, of respective degrees $\alpha_i$ and $\beta_i$, with $\alpha_i+\beta_i=n$. 
As $a_i,b_i$ are in $A_+$, we necessarily have $\alpha_i,\beta_i\geq 1$, so $\alpha_i,\beta_i<n$. By the induction hypothesis,
they belong to $A'$, so as $V\subseteq A_+$, $a\in A_+$. 
 \end{proof}

 \begin{prop}\label{prop2.3}
Let $A$ be a connected, graded filtered, locally finite algebra. If $\gr(A)$ is free as a graded algebra, then $A$ is free as a graded algebra.
\end{prop}

\begin{proof}
Let us put for any $n,k$, $\displaystyle \gr(A)_n^{(k)}=\frac{A_n^{(k)}}{A_n^{(k+1)}}$.
Then $\gr(A)$ is the direct sum of the subspaces $\gr(A)_n^{(k)}$. Let us choose a subspace $W$ of $\gr(A)$ such that $\gr(A)_+=W\oplus \gr(A)_+^2$ and
\[W=\bigoplus_{k,n\in \N}W\cap \gr(A)_n^{(k)}.\]
The product of $\gr(A)$ will be denoted by $*$ in the sequel of this proof. By Proposition \ref{prop2.2}, $W$ freely generates $(\gr(A),*)$
Let us choose a basis $(w_i)_{i\in I}$ such that for any $i\in I$, $w_i\in \gr(A)_{n_i}^{(k_i)}$ for a certain $(k_i,n_i)\in \N$. We put then $w_i=\pi_A^{(k_i)}(v_i)$,
with $v_i\in A_{n_i}^{(k_i)}$. Let us prove that $(v_{i_1}\ldots v_{i_p})_{p\geq 0,\: i_1,\ldots,i_p\in I}$ is a basis of $A$
(which will prove that $A$ is freely generated by the elements $v_i$). 

Let $(a_{i_1,\ldots,i_p})_{k\geq 0,\: i_1,\ldots,i_p\in I}$ be a family of scalars with a finite, but not zero, number of non zero elements, such that
\[\sum_{p\in \N}\sum_{i_1,\ldots,i_p\in I} a_{i_1,\ldots,i_p}v_{i_1}\ldots v_{i_p}=0.\]
Let us put $l=\min\{k_{i_1}+\ldots+k_{i_p}\mid a_{i_1,\ldots,i_p}\neq 0\}$. Then 
\[0=\pi_A^{(l)}\left(\sum_{p\in \N}\sum_{i_1,\ldots,i_p\in I} a_{i_1,\ldots,i_p}v_{i_1}\ldots v_{i_p}\right)
=\sum_{\substack{i_1,\ldots,i_p\in I,\\ k_{i_1}+\ldots+k_{i_p}=l}} a_{i_1,\ldots,i_p}w_{i_1}*\ldots* w_{i_p}.\]
As $\gr(A)$ is freely generated by $W$, the coefficients $a_{i_1,\ldots,i_p}$, with $k_{i_1}+\ldots+k_{i_p}=l$,
are all zeros: this contradicts the definition of $l$. So $(v_{i_1}\ldots v_{i_p})_{p\geq 0,\: i_1,\ldots,i_p\in I}$ is linearly independent.
To prove that this is a basis, it is equivalent to prove that the elements $v_i$ generates the algebra $A$.
Let us denote by $A'$ the subalgebra of $A$ generated by these elements $v_i$ and let us prove that for any $n\in \N$,
for any $x\in A_n$, $x\in A'$. 
There exists a scalar $k$ such that $x\in A_n^{(k)}$:
we proceed by decreasing induction on $k$. If $k\geq K(n)$, as $A_n^{(K(n))}=(0)$, $x=0$ and the result is obvious.
Let us assume the result at all ranks $>k$. As $W$ generates $\gr(A)$, we can write
\[\pi_A^{(k)}(x)=\sum_{p\in \N} \sum_{i_1,\ldots,i_p\in I} b_{i_1,\ldots i_p}w_{i_1}*\ldots * w_{i_k}.\]
Then
\[x-\sum_{p\in \N}\sum_{i_1,\ldots,i_p\in I} b_{i_1,\ldots i_p}v_{i_1}\ldots w_{i_k} \in A_n^{(k+1)}.\]
By the induction hypothesis on $k$, $x\in A'$. 
\end{proof}

\begin{remark}
The conditions of connectedness and of local finiteness seem to be necessary. Here is an example. Consider $A=\K[[X]]$, with the filtration given by $A^{(k)}=X^k\K[[X]]$.
Then $\gr(A)$ is isomorphic to $\K[X]$, so is free, whereas $A$ is not. 
\end{remark}

\subsection{Free enveloping algebras}

\begin{prop} \label{prop2.4}
Let $\g$ be a Lie algebra and $V$ be a subspace of $\g$. The following conditions are equivalent:
\begin{enumerate}
\item $\g$ is freely generated, as a Lie algebra, by $V$.
\item $\calU(\g)$ is freely generated, as an algebra, by $V$.
\end{enumerate}
\end{prop}

\begin{proof}
$\Longrightarrow$. Let $A$ be an algebra and let $f:V\longrightarrow A$ be a linear map. 
Let us prove that there exists a unique algebra map $F:\calU(\g)\longrightarrow A$ such that $F_{\mid V}=f$. 
\textit{Existence.} As $\g$ is freely generated by $V$, there exists a (unique) Lie algebra map 
$\tilde{f}:\g\longrightarrow A$ such that $\tilde{f}_{\mid V}=f$. By the universal property of $\calU(\g)$, 
there exists a (unique) algebra map $F:\calU(\g)\longrightarrow A$ such that $F_{\mid \g}=\tilde{f}$. Hence, $F_{\mid V}=f$. 
\textit{Unicity}. Let $F':\mathcal{U}\longrightarrow A$ be an algebra morphism such that $F'_{\mid V}=f$. As $\g$ is generated by $V$, 
$F'_{\mid \g}=F_{\mid \g}$. As $\g$ generates $\calU(\g)$ as an algebra, $F'=F$.\\

$\Longleftarrow$. Let $\h$ be a Lie algebra and let $f:V\longrightarrow \h$ be a linear map. 
Let us prove that there exists a unique Lie algebra map $F:\g\longrightarrow \h$
such that $F_{\mid V}=f$. \textit{Existence}. By the universal property of $\calU(\g)$, there exists a unique algebra map 
$\tilde{F}:\calU(\g)\longrightarrow \calU(\h)$, 
such that $\tilde{F}_{\mid V}=f$. As $V\subseteq \h$, for any $v\in V$,
\[\Delta\circ \tilde{F}(v)=\tilde{F}(v)\otimes 1+1\otimes \tilde{F}(v)=(\tilde{F}\otimes \tilde{F})\circ \Delta(v).\]
As $V$ generates $\calU(\g)$, $\tilde{F}$ is a bialgebra morphism. Therefore, it induces a Lie algebra morphism from 
$\prim(\calU(\g))=\g$ to  $\prim(\calU(\h))=\h$, which restriction $F$ to $V$ is $f$. 
\textit{Unicity}. Let $F':\g \longrightarrow\h$ be another Lie algebra morphism such that $F'_{\mid V}=f$. 
By the universal property of $\calU(\g)$, $F'$ is extended as an algebra morphism $\tilde{F}'$
from $\calU(\g)$ to $\calU(\h)$. By construction, $\tilde{F}'_{\mid V}=f$. As $V$ generates $\calU(\g)$, $\tilde{F'}=\tilde{F}$. 
Therefore, $F'=\tilde{F}'_{\mid \g}=F_{\mid \g}=F$.
\end{proof}

\begin{lemma}\label{lem2.5}
Let $\g$ be a Lie algebra. Then, in $\calU(\g)$, $\g\cap \ker(\varepsilon)^2=[\g,\g]$.
\end{lemma}

\begin{proof} 
$\supseteq$: immediate. $\subseteq$: let $\phi:\calU(\g)\longrightarrow \calU\left(\dfrac{\g}{[\g,\g]}\right)$ be the bialgebra morphism,
unique extension of the canonical surjection from $\g$ to $\dfrac{\g}{[\g,\g]}$. As $\dfrac{\g}{[\g,\g]}$ is an abelian Lie algebra, $\calU\left(\dfrac{\g}{[\g,\g]}\right)$
is the symmetric algebra $S\left(\dfrac{\g}{[\g,\g]}\right)$. Let $x\in \g\cap \ker(\varepsilon)^2$. Then $\phi(x)\in \dfrac{\g}{[\g,\g]}\cap \ker(\varepsilon)^2=(0)$,
in the symmetric algebra $S\left(\dfrac{\g}{[\g,\g]}\right)$. So $x\in \ker(\phi)\cap\g=[\g,\g]$. 
\end{proof}

\begin{cor}\label{cor2.6}
Let $\g$ be a graded, free Lie algebra with $\g_0=(0)$ and let $V\subseteq \g$ be a graded subspace such that 
$\g=V\oplus [\g,\g]$. Then $V$ freely generates $\g$.
\end{cor}

\begin{proof}
By Proposition \ref{prop2.4}, we know that $\calU(\g)$ is a free algebra. By the same proposition, 
it is enough to prove that $V$ freely generates $\calU(\g)$. As $\g$ is graded, $\calU(\g)$ is also graded, and connected as $\g_0=(0)$. 
By Proposition \ref{prop2.2}, it is enough to prove that $\ker(\varepsilon)=V\oplus \ker(\varepsilon)^2$.
By Lemma \ref{lem2.5}, $\g\cap \ker(\varepsilon)^2=[\g,\g]$. As $\ker(\varepsilon)=\g+\ker(\varepsilon)^2$, we obtain the result. 
\end{proof}

\section{Filtration of a graded connected bialgebra}

\subsection{Counital filtration of a graded connected bialgebra}

\begin{prop}\label{prop3.1}
Let $H$ be a graded connected bialgebra. We put, for all $k\in \N$, $H^{(k)}=\ker(\varepsilon)^k$. 
This makes $H$ a graded filtered locally finite bialgebra.
This filtration is called the counital filtration of $H$. 
\end{prop}

\begin{proof}
Obviously, $H^{(0)}=H$; for any $k\in \N$, $H^{(k+1)}\subseteq H^{(k)}$; for any $k,l\in \N$, $m(H^{(k)}\otimes H^{(l)})\subseteq H^{(k+l)}$; 
and for any $n\geq 1$, $\varepsilon\left(H^{(k)}\right)=(0)$. 
For any $k$, $H^{(k)}$ is the subspace of $H$ generated by the products $x_1\ldots x_k$, with $x_i\in \ker(\varepsilon)$ for any $i$.
As $\displaystyle \ker(\varepsilon)=\bigoplus_{n\geq 1} H_n$, we obtain by multilinearity that $H^{(k)}$ 
is generated by the products $x_1\ldots x_k$, such that for any $i$, $x_i$ is homogeneous of degree $\geq 1$. By homogeneity of the product, 
we obtain that for any $k\in \N$, 
\[H^{(k)}=\bigoplus_{n=0}^\infty H^{(k)}\cap H_n.\]
We put $H_n^{(k)}=H^{(k)}\cap H_n$. Then $H^{(k)}_n$ is generated  by the products $x_1\ldots x_k$, such that for any $i$, $x_i$ is homogeneous of degree $n_i\geq 1$, 
with $n_1+\ldots+n_k=n$. If $k>n$, there is no such products, so $H_n^{(k)}=(0)$, and $H$ is locally finite. It also implies that for any $k\in \N$,
\[H^{(k)}\subseteq \bigoplus_{n\geq H} H_n,\]
which implies that  $\displaystyle \bigcap_{k\in \N} H_k=(0)$. 

Let us prove that for any $n\in \N$, $\displaystyle \Delta\left(H^{(n)}\right)\subseteq \sum_{k+l=n}H^{(k)}\otimes H^{(l)}$.. The result is obvious if $k=0$. If $x\in \ker(\varepsilon)$, then
\[\Delta(x)-x\otimes 1-1\otimes x\in \ker(\varepsilon)\otimes \ker(\varepsilon).\]
so
\[\Delta(x)-1\otimes x-x\otimes 1\in H^{(1)}\otimes H^{(1)}.\]
Therefore, if $x_1,\ldots,x_n \in \ker(\varepsilon)$,
\begin{align}
\label{eq1}
\Delta(x_1\ldots x_n)-\sum_{I\subseteq [n]} \left(\prod_{i\in I} x_i\right)\otimes \left(\prod_{i\notin I} x_i\right)\in \sum_{k+l>n} H^{(k)}\otimes H^{(l)},
\end{align}
which gives the result.
\end{proof}

As a consequence, $H$ is a bialgebra in the category of graded filtered spaces, and is locally finite as a graded filtered space. 
Using the functor $\gr$ and its compatibility with the tensor products, $\gr(H)$ is a graded bialgebra. Its product is denoted by $m^{gr}$ and its coproduct by $\Delta^{gr}$. 

\begin{prop}\label{prop3.2}
If $H$ is a graded connected bialgebra, free as an algebra. We give it its counital filtration. Then $(H,m)$ and $(\gr(H),m^{gr})$ are isomorphic as graded algebras.
\end{prop}

\begin{proof}
Let us choose a complement $V_n$ of $H^{(2)}_n$ into $H^{(1)}_n$. By definition of $H^{(2)}$ and $H^{(1)}$, 
putting $V=\displaystyle \bigoplus_{n=1}^\infty V_n$, then $\ker(\varepsilon)=V\oplus \ker(\varepsilon)^2$, so $V$ freely generates $H$ as an algebra by Proposition \ref{prop2.2}. 
Let us choose a basis $(v_i)_{i\in I}$ of $V$ made of homogeneous elements.
A basis of $H^{(k)}$ is given by $(v_{i_1}\ldots v_{i_l})_{l\geq k, i_1,\ldots,i_l\in I}$. Therefore, a basis of $\gr(H)$ is given by 
$(\pi_H^{(k)}(v_{i_1}\ldots v_{i_k}))_{k\geq 0,\: i_1,\ldots,i_k\in I}$.
The unique algebra morphism from $H$ to $\gr(H)$ (implied by the freeness of $H$) sending $v_i$ to $\pi_H^{(1)}(v_i)$ for any $i\in I$ sends the basis 
$(v_{i_1}\ldots v_{i_k})_{k\geq 0,\: i_1,\ldots,i_k\in I}$ to the basis $(\pi_H^{(k)}(v_{i_1}\ldots v_{i_k}))_{k\geq 0,\: i_1,\ldots,i_k\in I}$, so it is an isomorphism. 
\end{proof}

\begin{prop}\label{prop3.3}
Let $H$ is a graded connected bialgebra. We give it its counital filtration. Then $\gr(H)$ is cocommutative. 
\end{prop}

\begin{proof}
Let $x\in \dfrac{H^{(k)}}{H^{(k+1)}}$. We can assume that $x=\pi_H^{(k)}(x_1\ldots x_k)$, with $x_1,\ldots,x_k\in \ker(\varepsilon)$. By (\ref{eq1}), we obtain that
\begin{align*}
\Delta^{gr}(x)&=\pi_{H\otimes H}^{(k)}(\Delta(x_1\ldots x_n))\\
&=\pi_{H\otimes H}^{(k)}\left(\sum_{I\subseteq [n]} \left(\prod_{i\in I} x_i\right)\otimes \left(\prod_{i\notin I} x_i\right)\right)+0\\
&=\sum_{I\subseteq [n]} \pi_H^{(|I|)} \left(\prod_{i\in I} x_i\right)\otimes \pi_H^{(k-|I|)}\left(\prod_{i\notin I} x_i\right).
\end{align*}
It is clearly cocommutative. \end{proof}

\subsection{Graded filtered connected Lie algebras}

\begin{lemma}\label{lem3.4}
Let $H$ be a graded filtered bialgebra. For any $k\in \N$, we put $\prim(H)^{(k)}=\g\cap H^{(k)}$.
This makes $\prim(H)$ a graded filtered Lie algebra. Moreover, $\gr(\prim(H))$ is a Lie subalgebra of $\prim(\gr(H))$.
\end{lemma}

\begin{proof}
It is immediate that $(\prim(H)^{(k)})_{k\in \N}$ is a filtration of $\prim(H)$. Moreover, there is a canonical injection
\[\gr(\prim(H))=\bigoplus_{k=0}^\infty \frac{\prim(H)\cap H^{(k)}}{\prim(H)\cap H^{(k+1)}}\subseteq \bigoplus_{k=0}^\infty \frac{H^{(k)}}{H^{(k+1)}}=\gr(H).\]
Let $x\in \prim(H)\cap H^{(k)}$. Then
\begin{align*}
\Delta^{gr}\left(\pi_H^{(k)}(x)\right)&=\pi_{H\otimes H}^{(k)}(x\otimes 1_H+1_H\otimes x)\\
&=\pi_H^{(k)}(x)\otimes \pi_H^{(0)}(1_H)+\pi_H^{(0)}(1_H)\otimes \pi_H^{(k)}(x)\\
&=\pi_H^{(k)}(x)\otimes 1_{\gr(H)}+1_{\gr(H)}\otimes \pi_H^{(k)}(x),
\end{align*}
so $\pi_H^{(k)}(x)\in \prim(\gr(H))$: we obtain that $\gr(\prim(H))\subseteq \prim(\gr(H))$.\\

We denote by $\{-,-\}$ the Lie bracket of $\gr(\prim(H))$ and by $\{-,-\}'$ the Lie bracket of $\prim(\gr(H))$. 
Let $x\in \prim(H)\cap H^{(k)}$,  and $y\in \prim(H)\cap H^{(l)}$. 
\begin{align*}
\left\{\pi_H^{(k)}(x),\pi_H^{(l)}(y)\right\}&=\pi_H^{(k+l)}(xy-yx)\\
&=\pi_H^{(k+l)}(xy)-\pi_H^{(k+l)}(yx)\\
&=\pi_H^{(k)}(x)\pi_H^{(l)}(y)-\pi_H^{(l)}(y)\pi_H^{(k)}(x)\\
&=\left\{\pi_H^{(k)}(x),\pi_H^{(l)}(y)\right\}',
\end{align*}
so $\gr(\prim(H))$ is a Lie subalgebra of $\prim(\gr(H))$.
\end{proof}

\begin{prop}\label{prop3.5}
Let $\g$ be a graded filtered Lie algebra, connected and locally finite. If $\gr(\g)$ is a free Lie algebra, then $\g$ is a free Lie algebra.  
\end{prop}

\begin{proof}
For any $k,n\in \N$, we put
\[\calU(\g)_n^{(k)}=\mathrm{Vect}\left(x_1\ldots x_p\mid p\in \N,\: x_j\in \g_{n_j}^{(k_j)},\: \sum_{j=1}^p n_j=n,\:\sum_{j=1}^p k_j=k\right).\]
This makes $\calU(\g)$ a graded filtered bialgebra. Moreover, as $\g^{(0)}=(0)$, $\calU(\g)_0=\K 1_{\calU(\g)}$, $\calU(\g)$ is connected.
Let $n\in \N$. We put
\[K'(n)=\max\left\{K(j_1)+\ldots+K(j_p)\mid p\in \N,\:j_1,\ldots,j_p\geq 1, \:j_1+\ldots+j_p=n\right\}.\]
Let $x_1\ldots x_p$ be a generator of $\calU(\g)^{(K'(n))}$, with $x_i\in \g_{n_j}^{(k_j)}$. If there exists $j$ such that $n_j=0$, 
then as $\g$ is connected, $x_j=0$. Otherwise, there exists $j$ such that $k_j\geq K(n_j)$, by definition of $K'(n)$.
Then $x_j=0$. Hence, $\calU(\g)_n^{(K'(n))}=(0)$: $\calU(\g)$ is a locally finite connected graded filtered bialgebra. \\

Let us prove that for any $k$, $\g^{(k)}\cap \calU(\g)^{(k+1)}=\g^{(k+1)}$. The inclusion $\supseteq$ is obvious. 
Let us consider $x\in \calU(\g)^{(k+1)}\cap \g^{(k)}$. As $x\in \calU(\g)^{(k+1)}$, $x$ is a noncommutative polynomial in
elements $x_1,\ldots,x_p\in \g$, with complementary conditions about the filtration that we leave to the reader. 
As it belongs to $\g$, it is a Lie polynomial, that is to say a span of iterated brackets of elements of $\g$.
The condition about the filtration and the fact that the bracket respects the filtration gives that $x\in \g^{(k+1)}$.\\

By Lemma \ref{lem3.4} applied to the bialgebra $\calU(\g)$, as $\prim(\calU(\g))=\g$, we obtain that $\gr(\g)$ is a Lie subalgebra of $\prim(\gr(\calU(\g)))$.
This induces a bialgebra morphism
\[\Phi:\calU(\gr(\g))\longrightarrow \gr(\calU(\g)),\]
which is the identity on $\gr(\g)$. Let us assume that it is not injective, and let us consider $x\in \ker(\Phi)$, of minimal degree $n$.
By minimality of $n$, $x$ is necessarily primitive, so belongs to $\prim(\calU(\gr(\g)))=\gr(\g)$, which implies that $\Phi(x)=x=0$:
this is a contradiction. So $\Phi$ is injective. Let us now prove that $\Phi$ is surjective: we consider $x\in \gr(\calU(\g))_n$
and proves that it belongs to $\im(\Phi)$ by induction on $n$. If $n=0$, then $x=\lambda 1_{\gr(\calU(\g))}$ for some $\lambda\in \K$,
so belongs to $\im(\Phi)$. Let us assume the results at all ranks $<n$. Let $x\in \gr(\calU(\g))_n^{(k)}$, we proceed
by decreasing induction on $k$. If $k\geq K(n)$, then $x=0$. Let us assume the result at all ranks $>k$. 
We can restrict ourselves to $x=\pi_{\calU(\g)}^{(k)}(x_1\ldots x_p)$, with $x_i\in \g_{n_i}^{(k_i)}$, $n_1+\ldots+n_p=n$,
$k_1+\ldots+k_p=p$. Then
\[x-\underbrace{\pi_{\calU(\g)}^{(k_1)}(x_1)\ldots \pi_{\calU(\g)}^{(k_p)}(x_p)}_{\in \im(\Phi)}\in \gr(\calU(\g))_n^{(k+1)}.\]
By the induction hypothesis on $k$, $x\in \im(\Phi)$. 
We finally proved that the bialgebras $\gr(\calU(\g))$ and $\calU(\gr(\g))$ are isomorphic. By hypothesis, $\gr(\calU(\g))$
is a free graded algebra.\\

Let $V$ be a graded subspace of of $\gr(\g)$ such that $\gr(\g)=V\oplus [\gr(\g),\gr(\g)]$. By Corollary \ref{cor2.6}, $V$ freely generates
$\gr(\calU(\g))$. By Proposition \ref{prop2.3}, $\calU(\g)$ is a graded free algebra.  By proposition \ref{prop2.4}, $\g$ is a graded free Lie algebra. \end{proof}

\section{Main result}

\begin{theo}\label{theo4.1}
Let $H$ be a graded connected bialgebra, free as an algebra. Then the Lie algebra $\prim(H)$ is free.
\end{theo}

\begin{proof}
\textit{First case}. Let us assume that $H$ is cocommutative. By Cartier--Quillen--Minor--Moore's theorem, $H$ is isomorphic to $\calU(\prim(H))$, so is generated by its primitive elements.
Hence, there exists a graded subspace $V$ of $\prim(H)$, such that $\ker(\varepsilon)=V\oplus \ker(\varepsilon)^2$. As a consequence, $V$ freely generates $H$ by Proposition \ref{prop2.2}. 
By Proposition \ref{prop2.4}, $V$ freely generates the Lie algebra $\prim(H)$. \\

\textit{General case}. We give $H$ its counital filtration $(H^{(n)})_{n\in \N}$, see Proposition \ref{prop3.1}. Then $\gr(H)$ is also free by Proposition \ref{prop3.2}.
Moreover, it is cocommutative by Proposition \ref{prop3.3}: by the first case, $\prim(\gr(H))$ is a free Lie algebra. 

For any $k\in \N$, we put $\prim(H)^{(k)}=\prim(H)\cap H^{(k)}$. Proposition \ref{prop3.1} implies that this defines a locally finite filtration of $\prim(H)$,
and by Lemma \ref{lem3.4}, $\gr(\prim(H))$ is a Lie subalgebra of $\prim(\gr(H))$.  The latter is a free Lie algebra. By Shirshov and Witt's theorem \cite{Shirshov53,Witt56}, 
$\gr(\prim(H))$ is free. By Proposition \ref{prop3.5}, $\prim(H)$ is free.  \end{proof}

\bibliographystyle{amsplain}
\bibliography{biblio}

\providecommand{\bysame}{\leavevmode\hbox to3em{\hrulefill}\thinspace}
\providecommand{\MR}{\relax\ifhmode\unskip\space\fi MR }
\providecommand{\MRhref}[2]{%
  \href{http://www.ams.org/mathscinet-getitem?mr=#1}{#2}
}
\providecommand{\href}[2]{#2}
\begin{thebibliography}{10}

\bibitem{Abe1980}
Eiichi Abe, \emph{Hopf algebras}, Cambridge Tracts in Mathematics, vol.~74,
  Cambridge University Press, Cambridge-New York, 1980, Translated from the
  Japanese by Hisae Kinoshita and Hiroko Tanaka.

\bibitem{Aguiar2008}
Marcelo Aguiar and Rosa~C. Orellana, \emph{The {Hopf} algebra of uniform block
  permutations.}, J. Algebr. Comb. \textbf{28} (2008), no.~1, 115--138
  (English).

\bibitem{Aguiar2006}
Marcelo Aguiar and Frank Sottile, \emph{Structure of the {L}oday-{R}onco {H}opf
  algebra of trees}, J. Algebra \textbf{295} (2006), no.~2, 473--511.

\bibitem{Bergeron2023}
Nantel Bergeron, Rafael~S. Gonz{\'a}lez~d'Le{\'o}n, Shu~Xiao Li, C.~Y.~Amy
  Pang, and Yannic Vargas, \emph{Hopf algebras of parking functions and
  decorated planar trees}, Adv. Appl. Math. \textbf{143} (2023), 62 (English),
  Id/No 102436.

\bibitem{Cartier2021}
Pierre {Cartier} and Fr\'ed\'eric {Patras}, \emph{{Classical Hopf algebras and
  their applications}}, vol.~29, Cham: Springer, 2021 (English).

\bibitem{Catoire2022}
Pierre Catoire, \emph{Tridendriform structures}, arXiv:2207.03839, 2022.

\bibitem{Foissy25}
G\'erard H.~E. Duchamp, Lo{\"\i}c Foissy, Nguyen Hoang-Nghia, Dominique
  Manchon, and Adrian Tanasa, \emph{A combinatorial non-commutative {H}opf
  algebra of graphs}, Discrete Math. Theor. Comput. Sci. \textbf{16} (2014),
  no.~1, 355--370.

\bibitem{Foissy1}
Lo{\"\i}c Foissy, \emph{Les alg\`ebres de {H}opf des arbres enracin\'es
  d\'ecor\'es. {I}}, Bull. Sci. Math. \textbf{126} (2002), no.~3, 193--239.

\bibitem{Foissy2}
\bysame, \emph{Les alg\`ebres de {H}opf des arbres enracin\'es d\'ecor\'es.
  {II}}, Bull. Sci. Math. \textbf{126} (2002), no.~4, 249--288.

\bibitem{Foissy5}
\bysame, \emph{Bidendriform bialgebras, trees, and free quasi-symmetric
  functions}, J. Pure Appl. Algebra \textbf{209} (2007), no.~2, 439--459.

\bibitem{Foissy15}
\bysame, \emph{Primitive elements of the {H}opf algebra of free quasi-symmetric
  functions}, Combinatorics and physics, Contemp. Math., vol. 539, Amer. Math.
  Soc., Providence, RI, 2011, pp.~79--88.

\bibitem{Foissy17}
\bysame, \emph{Ordered forests and parking functions}, Int. Math. Res. Not.
  IMRN (2012), no.~7, 1603--1633.

\bibitem{Foissy18}
\bysame, \emph{Algebraic structures on double and plane posets}, J. Algebraic
  Combin. \textbf{37} (2013), no.~1, 39--66.

\bibitem{Foissy20}
\bysame, \emph{Plane posets, special posets, and permutations}, Adv. Math.
  \textbf{240} (2013), 24--60.

\bibitem{Foissy27}
Lo{\"\i}c Foissy and Claudia Malvenuto, \emph{The {H}opf algebra of finite
  topologies and {$T$}-partitions}, J. Algebra \textbf{438} (2015), 130--169.

\bibitem{Gelfand1995}
Israel~M. Gelfand, Daniel Krob, Alain Lascoux, Bernard Leclerc, Vladimir~S.
  Retakh, and Jean-Yves Thibon, \emph{Noncommutative symmetric functions}, Adv.
  Math. \textbf{112} (1995), no.~2, 218--348.

\bibitem{Hazewinkel2005}
Michiel Hazewinkel, \emph{Symmetric functions, noncommutative symmetric
  functions and quasisymmetric functions. {II}}, Acta Appl. Math. \textbf{85}
  (2005), no.~1-3, 319--340.

\bibitem{Holtkamp2003}
Ralf Holtkamp, \emph{Comparison of {Hopf} algebras on trees.}, Arch. Math.
  \textbf{80} (2003), no.~4, 368--383 (English).

\bibitem{Kashaev2023}
Rinat Kashaev, \emph{A course on {Hopf} algebras}, Universitext, Cham:
  Springer, 2023 (English).

\bibitem{Loday1998}
Jean-Louis Loday and Mar\'{\i}a Ronco, \emph{Hopf algebra of the planar binary
  trees}, Adv. Math. \textbf{139} (1998), no.~2, 293--309.

\bibitem{Malvenuto1995}
Claudia Malvenuto and Christophe Reutenauer, \emph{Duality between
  quasi-symmetric functions and the {S}olomon descent algebra}, J. Algebra
  \textbf{177} (1995), no.~3, 967--982.

\bibitem{Malvenuto2011}
\bysame, \emph{A self paired {H}opf algebra on double posets and a
  {L}ittlewood-{R}ichardson rule}, J. Combin. Theory Ser. A \textbf{118}
  (2011), no.~4, 1322--1333.

\bibitem{Manchon1997}
Dominique Manchon, \emph{Bitensorial {Hopf} algebra}, Commun. Algebra
  \textbf{25} (1997), no.~5, 1537--1551 (French).

\bibitem{MilnorMoore65}
John~W. Milnor and J.~C. Moore, \emph{On the structure of {Hopf} algebras},
  Ann. Math. (2) \textbf{81} (1965), 211--264.

\bibitem{Novelli2006}
Jean-Christophe {Novelli} and Jean-Yves {Thibon}, \emph{{Construction de
  trig\`ebres dendriformes.}}, {C. R., Math., Acad. Sci. Paris} \textbf{342}
  (2006), no.~6, 365--369 (French).

\bibitem{Novelli2006-2}
\bysame, \emph{Polynomial realization of some trialgebras}, FPSAC'06 (San
  Diego) (2006).

\bibitem{Novelli2007}
Jean-Christophe Novelli and Jean-Yves Thibon, \emph{Hopf algebras and
  dendriform structures arising from parking functions.}, Fundam. Math.
  \textbf{193} (2007), no.~3, 189--241 (English).

\bibitem{Ronco2000}
Mar\'{\i}a Ronco, \emph{Primitive elements in a free dendriform algebra}, New
  trends in Hopf algebra theory. Proceedings of the colloquium on quantum
  groups and Hopf algebras, La Falda, Sierras de C\'ordoba, Argentina, August
  9--13, 1999, Providence, RI: American Mathematical Society (AMS), 2000,
  pp.~245--263.

\bibitem{Shirshov53}
A.~I. Shirshov, \emph{Subalgebras of free {Lie} algebras}, Mat. Sb., Nov. Ser.
  \textbf{33} (1953), 441--452 (Russian).

\bibitem{Sweedler1969}
Moss~E. Sweedler, \emph{Hopf algebras}, Mathematics Lecture Note Series, W. A.
  Benjamin, Inc., New York, 1969.

\bibitem{Witt56}
Ernst Witt, \emph{Die {Unterringe} der freien {Lieschen} {Ringe}}, Math. Z.
  \textbf{64} (1956), 195--216 (German).

\end{thebibliography}

\end{document}